\documentclass[12pt]{amsart}
\usepackage[table,dvipsnames]{xcolor}
\usepackage{amsmath, amsthm, amscd, amsfonts, amssymb, caption, float, graphics, graphicx, hyperref, mathrsfs, subcaption, textcomp, tikz, tikz-cd, verbatim, ytableau}
\usepackage[margin=1in]{geometry}
\usepackage[shortlabels]{enumitem}
\usepackage[T1]{fontenc}
\usepackage[nameinlink,noabbrev,capitalise]{cleveref}
\usepackage{subfiles}

\newtheorem{theorem}{Theorem}[section]
\newtheorem{proposition}[theorem]{Proposition}

\newtheorem{corollary}[theorem]{Corollary}
\newtheorem{lemma}[theorem]{Lemma}

\newtheorem{definition}[theorem]{Definition}

\theoremstyle{definition}

\newtheorem{remark}[theorem]{Remark}
\newtheorem{example}[theorem]{Example}

\definecolor{darkblue}{RGB}{0,51,204}
\hypersetup{colorlinks,linkcolor={darkblue},citecolor={darkblue},urlcolor={darkblue}}

\newcommand{\suchthat}{\;\ifnum\currentgrouptype=16 \middle\fi|\;}

\def\ZZ{\mathbb{Z}}
\def\CC{\mathbb{C}}
\def\PP{\mathbb{P}}
\def\RR{\mathbb{R}}
\def\Mat{\operatorname{Mat}}
\def\Gr{\operatorname{Gr}}
\def\rk{\operatorname{rk}}
\def\cl{\operatorname{cl}}
\def\min{\operatorname{min}}
\def\Span{\operatorname{Span}}
\def\Bound{\operatorname{Bound}}
\def\Br{\operatorname{Br}}
\def\tb{\operatorname{tb}}
\def\Aug{\operatorname{Aug}}

\def\Pos{\Pi^\circ}
\def\Fix{\operatorname{Fix}}
\def\codim{\operatorname{codim}}

\begin{document}

\title{Lagrangian cobordism of positroid links}

\author[J. Asplund]{Johan Asplund}
\address{Department of Mathematics, Stony Brook University, 100 Nicolls Road, Stony Brook, NY 11794}
\email{johan.asplund@stonybrook.edu}
\author[Y. Bae]{Youngjin Bae}
\address{Department of Mathematics, Incheon National University, Republic of Korea} \email{yjbae@inu.ac.kr}
\author[O. Capovilla-Searle]{Orsola Capovilla-Searle}
\address{Department of Mathematics, University of California Davis, Davis 95616}
\email{ocapovillasearle@ucdavis.edu}
\author[M. Castronovo]{Marco Castronovo}
\address{Department of Mathematics, Columbia University, 2990 Broadway, New York, NY 10027}
\email{marco.castronovo@columbia.edu}
\author[C. Leverson]{Caitlin Leverson}
\address{Mathematics Program, Bard College, 30 Campus Road, Annandale-on-Hudson, NY 12504} \email{cleverson@bard.edu}
\author[A. Wu]{Angela Wu}
\address{Department of Mathematics, Bucknell University, Lewisburg, PA 17837} \email{a.wu@bucknell.edu}


\begin{abstract}

Positroid strata of the complex Grassmannian can be realized as augmentation varieties of Legendrians called positroid links. We prove that the partial order on strata induced by Zariski closure also has a symplectic interpretation, given by exact Lagrangian cobordism.

\end{abstract}

\maketitle
\thispagestyle{empty}

\section{Introduction}\label{SecIntroduction}

Positroid varieties are irreducible subvarieties of the complex Grassmannian that were first introduced in the study of total positivity~\cite{Lusztig_1998,postnikov2006total, Rietsch_2006}, and Poisson geometry~\cite{Brown_Goodearl_Yakimov_2006}. Open positroid varieties provide a stratification of the complex Grassmannian, and they can be enumerated by a handful of different combinatorial objects.

Positroid varieties are known to admit cluster structures, which have also been found on the coordinate rings of many algebraic varieties arising in representation theory, including double Bruhat cells \cite{FZ}, double Bott--Samelson cells \cite{SW}, positroid strata \cite{GL2}, and certain Richardson strata \cite{CGGLSS, GLSS, GLS}. Geometrically, this allows one to think of such varieties as the result of gluing algebraic tori along birational mutation maps, and their coordinate rings carry bases whose structure constants are positive integers counting tropical curves \cite{FG, GHKK}.

Legendrian links in $\RR^3$ are smooth links that are everywhere tangent to the plane field $\ker(dz-ydx)$ which is called the standard contact structure of $\RR^3$. Their interpolating objects are exact Lagrangian cobordisms in the symplectization of $\RR^3$. Legendrian links and exact Lagrangian cobordisms between them can be studied via the general framework of symplectic field theory~\cite{eliashberg2000introduction} which aims to use counts of pseudoholomorphic curves to define invariants of contact manifolds and the symplectic cobordisms between them. One such invariant is the Chekanov--Eliashberg differential graded algebra associated to a Legendrian link $\Lambda$, whose homology is a Legendrian isotopy invariant \cite{chekanov2002differential,eliashberg2000introduction}.

In favorable circumstances, the space of augmentations of the Chekanov--Eliashberg dg-algebra forms an algebraic variety $\Aug(\Lambda)$. 
Exact Lagrangian fillings (cobordisms from the empty set to $\Lambda$) induce augmentations. 

Augmentations have been shown to be related to simple microlocal sheaves associated to $\Lambda$~\cite{STZ,STWZ,NRSSZ}, leading to the idea that augmentation varieties should have cluster structures, with torus charts corresponding to embedded exact Lagrangian fillings of $\Lambda$ and mutations arising from Lagrangian surgery \cite{polterovich1991the}. So far this idea has been explored mainly for Legendrian links in the standard contact $\RR^3$, see \cite{Casals_Weng} and references therein for the most recent results. This bridge between contact geometry and cluster algebras has been fruitful in both directions, having been instrumental in resolving long-standing conjectures on the abundance of Lagrangian fillings of Legendrian links \cite{casals2022infinitely} and on the existence of cluster structures on spaces of interest in representation theory \cite{CGGLSS}.

In this article we explore this idea further, predicating that when augmentation varieties have compactifications stratified by augmentation varieties of smaller dimension, then the Legendrian submanifolds corresponding to adjacent strata should be related by exact Lagrangian cobordisms. We establish this in the simplest class of examples: positroid strata of complex Grassmannians \cite{KLS}. 
It is known that all positroid strata are isomorphic to the moduli space of simple microlocal sheaves of certain Legendrian links $\Lambda$ in $\RR^3$ with framings (marked points) \cite{STWZ} and to augmentation varieties of $\Lambda$ \cite{CGGS,CGGS2}.
The top positroid stratum was one of the key motivating examples for the development of cluster algebras \cite{FZ, S}, while cluster structures on strata of lower dimension were described more recently \cite{GL2}.

\subsection{Result}

The positroid strata $\Pos\subset\Gr(k,n)$ of complex Grassmannians are disjoint Zariski locally closed sets, see \cref{DefPositroidStratum}. There is a distinguished class of Legendrian links $\Lambda_{\Pos}$ in the standard contact $\RR^3$, referred to as positroid links whose augmentation varieties are related to the strata by an algebraic isomorphism
\begin{equation}\label{eq:pos_aug}
    \Pos \cong \Aug(\Lambda_{\Pos}) \times (\CC^\ast)^{N(\Pos)},
\end{equation}
where $N(\Pos) \in \ZZ_{\geq 0}$ is a non-negative integer depending on $\Pos$, see \cref{SecPositroidLinks} for a precise statement. The positroid link $\Lambda_{\Pos}$ is the Legendrian $(-1)$-closure of a braid on $k$-strands associated to $\Pos$, known as a juggling braid (see \cref{def:positroid_link}). The positroid strata can be enumerated by bounded affine permutations (see \cref{SecBoundedAffinePermutations}) and for each pair of integers $1\leq k < n$, the set of positroid strata of $\Gr(k,n)$ is partially ordered by declaring $\Pos_f \leq \Pos_g$ if and only if $\Pos_f\subset \cl(\Pos_g)$. Our main result is the following.

\begin{theorem}[\cref{thm:main1} and \cref{thm:main2}]\label{ThmMain}
Given two comparable positroid strata $\Pos_f\leq \Pos_g$ in $\Gr(k,n)$, their associated Legendrian links $\Lambda_{\Pos_f}$ and $\Lambda_{\Pos_g}$ are related by an exact Lagrangian cobordism from $\Lambda_{\Pos_f}$ to $\Lambda_{\Pos_g}$ whose Euler characteristic is
\[
\dim(\Pos_g)-\dim(\Pos_f)+\#\Fix(g)-\#\Fix(f)
\]
where $\#\Fix$ denotes the number of fixed points (see \cref{def:cycles}).
\end{theorem}
\begin{remark}
    \begin{enumerate}
        \item We defer experts to \cref{thm:aug_braid_var_iso} and \cref{rmk:marked_pts} for a discussion on marked points placed on the positroid links.
        \item Two positroid links being exact Lagrangian cobordant does not imply that the corresponding positroid strata are comparable in the partial order, see \cref{ex:main_thm_not_equiv} and \cref{rmk:partial_converse}.
    \end{enumerate}
\end{remark}

Note that even for small $k$ and $n$ complex Grassmannians have many positroid strata, and their partial order is quite complicated, see \cref{ExGr24Strata} and \cref{fig:hasse_diagram_bound24}. The exact Lagrangian cobordism in \cref{ThmMain} is constructed by pinching contractible Reeb chords, which is a well-known technique in contact geometry. We establish that any chain connecting $\Pos_f$ and $\Pos_g$ in the partial order produces a sequence of pinch moves.

If $\Pos_f < \Pos_g$ then from \cref{ThmMain} there is an exact Lagrangian cobordism from $\Lambda_{\Pos_f}$ to $\Lambda_{\Pos_g}$ consisting of pinch moves. Let $r$ be the number of such moves. From \cite{Pan, gao2020augmentations} it follows that there is an open embedding relating the augmentation varieties of the ends:
$$\Aug(\Lambda_{\Pos_f})\times (\CC^*)^r \hookrightarrow \Aug(\Lambda_{\Pos_g}) .$$
This means that if the bounded affine permutations $f$ and $g$ are related by $r$ affine transpositions, under the identification between positroid strata and augmentation varieties in \eqref{eq:pos_aug} we get an open embedding
$$\Pos_f\times (\CC^*)^{r+N(\Pos_g)} \hookrightarrow \Pos_g\times (\CC^*)^{N(\Pos_f)} .$$
As pointed out to us by a referee it is an interesting question whether such embeddings admit a description purely in terms of algebraic combinatorics, i.e.\@ without using the connection with the topology of Legendrians.

\subsection*{Outline}
In \cref{SecBackground} we provide the necessary background on Legendrian links and exact Lagrangian cobordisms. In \cref{SecBoundedAffinePermutations} we provide the relevant definitions and properties of bounded affine permutations. We recall the definition of positroid strata of complex Grassmannians in \cref{SecPositroidStrata}.
In \cref{SecPositroidLinks} we describe positroid links via juggling braids coming from bounded affine permutations. Our main theorem \cref{ThmMain} is proven in \cref{SecCobordisms}.  In \cref{SecExamples} we discuss examples.

\subsection*{Acknowledgments}
	This project was initiated at American Institute of Mathematics (AIM) during the
	workshop ``Cluster algebras and braid varieties'' in January 2023. We are grateful
	to AIM and the workshop organizers Roger Casals, Mikhail Gorsky, Melissa
	Sherman-Bennett, and Jos\'{e} Simental. We thank James Hughes, Melissa Sherman-Bennett and Peng Zhou for helpful
	discussions. We thank Pavel Galashin for introducing us to bounded affine permutations. Finally, we thank Roger Casals and Pavel Galashin for comments on an earlier draft of the article. We are grateful to the anonymous referees for helpful comments and for the suggestion to use \cref{lma:inverse_bap} to shorten the proof of \cref{thm:main1}.
    JA was supported by the Knut and Alice Wallenberg Foundation.
    YB was supported by NRF grant NRF-2022R1C1C1005941.
    OCS was supported by NSF grant DMS-2103188.  MC was supported by an AMS-Simons
    Travel Grant. AW was supported by an AMS-Simons Travel
    Grant and NSF grant DMS-2238131.

\section{Background on contact geometry}\label{SecBackground}

In this section we briefly review some basic facts on
Legendrian links and exact Lagrangian cobordisms. See \cite{etnyre2022legendrian} for a more thorough introduction, and \cite{G} for background on contact geometry.

\subsection{Legendrian links}
	The \emph{standard contact structure} on $\RR^3$ is the plane field ${\xi := \ker(dz - ydx)}$. A smooth embedding of circles $\Lambda\subset\RR^3$ is called a \emph{Legendrian link} if $T_x \Lambda \subset \xi_x$ for all $x \in \Lambda$. Two Legendrian links are \emph{Legendrian isotopic} if they are smoothly isotopic through Legendrian links. The maps $\pi_L(x,y,z)=(x,y)$ and $\pi_F(x,y,z)=(x,z)$ are called the \emph{Lagrangian projection} and \emph{front projection}, respectively. The Lagrangian projection of a Legendrian link is an immersed curve with zero oriented area. The front projection of a Legendrian link does not have any vertical tangencies but instead has cusp and crossing singularities. Conversely, any immersed disjoint union of circles with cusp and crossing singularities and no vertical tangencies lifts uniquely to a Legendrian link in $\RR^3$, see \cref{fig:trefoil} for an example of both projections.
	\begin{figure}[!htb]
		\centering
		\raisebox{2mm}{\includegraphics{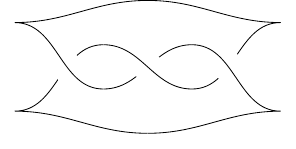}}
		\hspace{2cm}
		\includegraphics{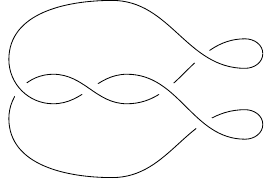}
		\caption{Front (left) and Lagrangian (right) projections of a Legendrian trefoil.}\label{fig:trefoil}
	\end{figure}

	The \emph{Thurston--Bennequin number} $\tb(\Lambda)\in\ZZ$ of a Legendrian link $\Lambda \subset \RR^3$ measures how much the contact structure $\xi$ rotates along $\Lambda$, and is defined as the linking number of $\Lambda$ and its push-off in any direction transverse to $\xi$. This number is easily computed from a front projection as
    \begin{align}\label{eq:thurston-bennequin}
        \tb(\Lambda) =& \,\#\text{\emph{positive crossings of }}\pi_F(\Lambda) - \#\text{\emph{negative crossings of }}\pi_F(\Lambda) \\
        &- \# \text{\emph{right cusps of }}\pi_F(\Lambda).\notag
    \end{align}

	A \emph{Reeb chord} of $\Lambda$ is a trajectory of the vector field $\partial_z$ that begins and ends on $\Lambda$. Note that the Lagrangian projection induces a bijection between Reeb chords and double points of $\pi_L(\Lambda)$. A Reeb chord of $\Lambda$ is \emph{contractible} if there exists a smooth homotopy of $\Lambda$ through Legendrian immersions (such that the Lagrangian projections only have transverse double points throughout the homotopy) that shrinks the length of the Reeb chord to zero, see \cite[Definition 6.13]{ekholm2016legendrian}.
 
\subsection{Exact Lagrangian cobordisms}
	The \emph{symplectization} of the standard contact $\RR^3$ is defined as $\RR \times \RR^3$ equipped with the closed non-degenerate 2-form $\omega = d \lambda$ where ${\lambda = e^t (dz -y dx)}$. A \emph{Lagrangian cobordism} from a Legendrian link $\Lambda_-$ to a Legendrian link $\Lambda_+$ is a smooth embedding of a surface $L\subset \RR\times\RR^3$ such that $\omega|_{TL}=0$ and such that $L$ is a cylinder over $\Lambda_{\pm}$ at infinity but is otherwise compact, i.e. there is some $T>0$ for which $L\cap [-T,T]$ is compact, 
	\begin{align*}
		\mathcal E_-(L) &:= 
  L \cap ((-\infty,-T) \times \RR^3) = (-\infty,-T) \times \Lambda_- \text{ and }\\
		\mathcal E_+(L) &:= 
  L \cap ((T,\infty) \times \RR^3) = (T,\infty) \times \Lambda_+.
	\end{align*}
	A Lagrangian cobordism $L$ is  \emph{exact} if there is a smooth function $f \colon L \to \RR$ such that $df = \lambda|_L$ and $f|_{\mathcal E_{\pm}(L)}$ is constant. A \emph{Lagrangian filling} is a Lagrangian cobordism with $\Lambda_- = \emptyset$. Exact Lagrangian cobordisms give a reflexive and transitive relation, but not a symmetric one \cite{chantraine2015lagrangian}.
 All known examples of exact Lagrangian cobordisms between Legendrians with maximal Thurston--Bennequin numbers arise from
	\begin{itemize}
		\item Legendrian isotopy,
		\item the unique exact Lagrangian disk filling of an unlinked unknot component with maximal Thurston--Bennequin number, and
		\item pinching a contractible Reeb chord.
	\end{itemize}
	The \emph{pinch move} is a local modification of $\Lambda\subset\RR^3$, depicted in \cref{fig:pinch}. A pinch move induces an exact Lagrangian cobordism in the symplectization of $\RR^3$ from the knot after a pinch move to the knot before the pinch move. When a pinch move is performed, the number of components of the Legendrian link either increases or decreases by one, so the resulting exact Lagrangian cobordism is topologically a pair of pants, and it is often called a saddle cobordism. See Figure~\ref{fig:pinch_trefoil_hopf} for an example of an exact Lagrangian saddle cobordism between two Legendrians.
	\begin{figure}[!htb]
		\centering
        \includegraphics{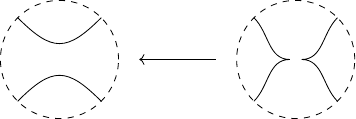}
        \hspace{2cm}
		\includegraphics{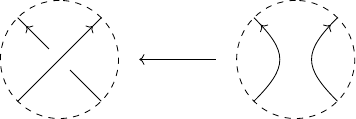}
		\caption{A pinch move in the front (left) and the Lagrangian (right) projection. The arrows show the direction of the induced exact Lagrangian cobordism.}\label{fig:pinch}
	\end{figure}

\section{Bounded affine permutations}\label{SecBoundedAffinePermutations}

In this section we review bounded affine permutations, the affine analogs of ordinary permutations. See \cite{KLS} for a more thorough introduction and an interpretation in terms of juggling. Throughout this section let $k,n\in\ZZ_{\geq 1}$ with $k\leq n$.

\begin{definition}\label{def:bounded_affine_perm}
An \emph{affine permutation of size $n$} is a bijection $f\colon\ZZ\to \ZZ$ satisfying $f(i+n)=f(i)+n$ for all $i\in\ZZ$. In addition, it is \emph{$k$-bounded} if
\begin{enumerate}
\item $i\leq f(i) \leq i+n$, and
\item $\sum_{i=1}^{n}(f(i)-i)=nk$.
\end{enumerate}
Denote the set of $k$-bounded affine permutations of size $n$ by $\Bound(k,n)$, and a $k$-bounded affine permutation $f:\ZZ\to \ZZ$ by $[f(1),f(2),\ldots,f(n)]$.
\end{definition}

\begin{lemma}\label{lma:inverse_bap}
    A bijection $f$ is a $k$-bounded affine permutation of size $n$ if and only if $g := -(-f)^{-1}$ is a $k$-bounded affine permutation of size $n$.
\end{lemma}
\begin{proof}
    Let $i\in \ZZ$ and define $j = -f(i)$. Then $f(i+n) = f(i)+n$ is equivalent to $i+n = -g(-f(i)-n)$ and hence $-g(j) + n = -g(j-n)$.
    Note that $i \leq f(i) \leq i+n$ is equivalent to $-g((-f)(i)) \leq f(i) \leq -g((-f)(i+n))$. We can rewrite these inequalities as $g(j) \geq j$ and $j\geq g(j-n) = g(j) - n$, which is equivalent to $j\leq g(j)\leq j+n$. Finally, $\sum_{i=1}^n (f(i)-i) = nk$ is equivalent to $\sum_{j=1}^n(-j+g(j)) = nk$.
\end{proof}

A bounded affine permutation $f\in \Bound(k,n)$ can be visualized in the plane as the set of line segments in $\RR^2$ from $(i,1)$ to $(f(i),0)$ for all $i\in\ZZ$, see \cref{fig:bap_example} for an example. Note that once $f(i+n)=f(i)+n$ for all $i\in\ZZ$, the picture is fully determined by the region in the red dashed box in \cref{fig:bap_example}.

\begin{figure}[!htb]
\includegraphics{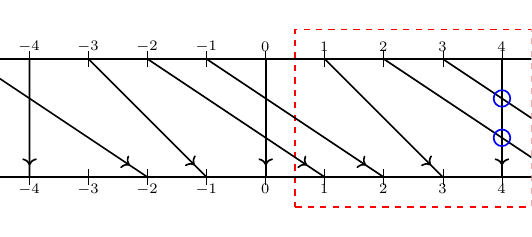}
\caption{The bounded affine permutation $f=[3,5,6,4]\in\Bound(2,4)$.}
\label{fig:bap_example}
\end{figure}

\begin{definition}\label{def:length}
For a bounded affine permutation $f\in \Bound(k,n)$, a pair $(i,j)\in \left\{1,\ldots,n\right\}^2$ is an \emph{affine inversion} if $i<j$ and either $f(i) > f(j)$ or $f(i) < f(j) - n$. The \emph{length} of an affine permutation $f$, $\ell(f)\in\ZZ_{\geq 0}$, is the number of affine inversions of $f$.
\end{definition}

\begin{example}
For the bounded affine permutation $f=[3,5,6,4]$ depicted in \cref{fig:bap_example}, $(2,4)$ and $(3,4)$ are the only affine inversions, thus $\ell(f)=2$. These two affine inversions correspond to the two circles in \cref{fig:bap_example}.
Similarly, we have $\ell([3,4,6,5]) = 1$ and $\ell([2,5,7,4])=3$.
\end{example}

We now equip $\Bound(k,n)$ with a partial order.

\begin{definition}\label{def:transposition}
An affine permutation $\sigma:\ZZ\to \ZZ$ of size $n$ is a \emph{transposition} if $\sigma(k)=\sigma_i(k)$ for some $i\in\ZZ$, where
\[
\sigma_i(k) := \begin{cases}
k+1,&\text{ if } k= i  \pmod n\\
k-1,&\text{ if } k= i+1 \pmod n\\
k,&\text{ if } k\neq i,i+1  \pmod n\\
\end{cases}.
\]
\end{definition}

\begin{definition}\label{def:partial_order}
Let $f,f' \in \Bound(k,n)$. Declare $f\lessdot f'$ if and only if $\ell(f)<\ell(f')$
and there exists an affine transposition $\sigma_i$ of size $n$ such that $f'=f\circ \sigma_i$ or $f'=\sigma_i \circ f$. Define a relation $<$ on $\Bound(k,n)$ as the transitive closure of the relation $\lessdot$.
\end{definition}\label{defn:partial_order}
It follows from the definition that $(\Bound(k,n),<)$ is a partially ordered set.

\begin{figure}[!htb]
\includegraphics{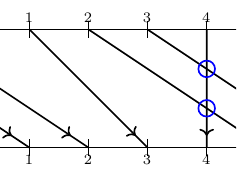}
\hspace{2cm}
\includegraphics{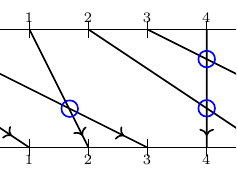}
\caption{Left: $[3,5,6,4]\in\Bound(2,4)$. Right: $[2,5,7,4] \in \Bound(2,4)$. }
\label{fig:bap_example_order}
\end{figure}

\begin{example}
Note that $\ell([3,4,6,5])<\ell([3,5,6,4])$ and $[3,5,6,4]=\sigma_4\circ[3,4,6,5]$, thus we have $[3,4,6,5] \lessdot [3,5,6,4]$. Similarly $\ell([3,5,6,4])<\ell([2,5,7,4])$ and $[2,5,7,4]=\sigma_2\circ[3,5,6,4]$, so that $[3,5,6,4]\lessdot[2,5,7,4]$, see \cref{fig:bap_example_order}. Then the induced partial order satisfies $[3,4,6,5] < [2,5,7,4]$. See \Cref{fig:hasse_diagram_bound24} for the Hasse diagram of the partial order $<$ on $\Bound(2,4)$.
\end{example}

\begin{figure}[!htb]
\includegraphics{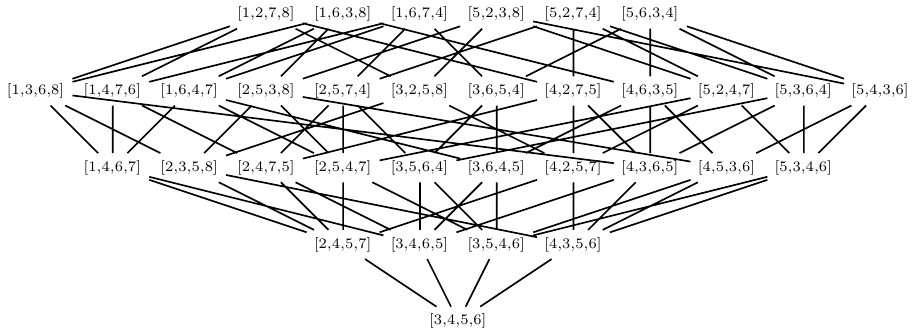}
\caption{The Hasse diagram of the partial order on $\Bound(2,4)$.}
\label{fig:hasse_diagram_bound24}
\end{figure}

As in the case of ordinary permutations, one can define cycles of bounded affine permutations.

\begin{lemma}
Let $f \in \Bound(k,n)$. Then $f$ induces a bijection $\bar{f}\colon \ZZ / n \to \ZZ / n$ defined by $\bar{f}([i])=[f(i)]$ for all $[i]\in \ZZ / n $.
\end{lemma}

\begin{proof}
Recall that as $f$ is a bounded affine permutation, for all $i,t\in\ZZ$, we know 
$f(i+tn)=f(i)+tn = f(i)\pmod n$.
Thus, $\bar{f}$ is well-defined. Moreover, we have a well-defined inverse function $\bar{f}^{-1}$ given by $\bar{f}^{-1} ([i]) = [f^{-1}(i)]$. So $\bar{f}$ is a bijection.
\end{proof}

\begin{definition}\label{def:cycles}
    A \emph{cycle of length $t$} of $f\in \Bound(k,n)$ is a tuple $(i_1, \ldots, i_t) \in (\ZZ/n)^t$ up to cyclic permutation such that 
    \[
    \bar{f} \colon i_1 \longmapsto i_2 \longmapsto \cdots \longmapsto i_t\longmapsto i_1,
    \]
    where $i_1,\ldots, i_t$ are all distinct. A cycle of length $1$ is called a \emph{fixed point} of $f$.
\end{definition}

\begin{example}
    The affine permutation $f=[3,5,6,4]$ has one cycle of length three being $(1,3,2)$ and one cycle of length one being $(4)$.
\end{example}

\section{Positroid strata}\label{SecPositroidStrata}

In this section, we collect definitions and known properties of positroid strata of the complex Grassmannian from the literature \cite{KLS,GL}.

\subsection{Complex Grassmannians}

Fix $k,n\in\ZZ_{\geq 1}$ such that $k \leq n$, and write $\Mat(k,n)$ for the set of $k\times n$ matrices with complex entries.

\begin{definition}\label{DefGrassmannian}
The \emph{complex Grassmannian of $k$-planes} in $\CC^n$ is
\[\Gr(k,n) = \{ M\in\Mat(k,n) \suchthat \rk(M)=k\} / \text{\emph{row operations}}.\]
\end{definition}

Complex Grassmannians are smooth projective varieties. A widely used projective embedding of $\Gr(k,n)$ in $\PP^{\binom nk-1}$ is 
given by Plücker coordinates, see \cite[Lecture 6]{harris1992algebraic}.

\begin{definition}
Given $M\in\Mat(k,n)$ with column vectors $M_1, \ldots,M_n$ and ${1\leq i_1<\ldots<i_k\leq n}$, we define the \emph{Plücker coordinate} $\Delta_{i_1,\ldots ,i_k}(M)$ to be 
\[\Delta_{i_1,\ldots ,i_k}(M)=\det\begin{bmatrix}M_{i_1},M_{i_2},M_{i_3},\ldots,M_{i_k}\end{bmatrix}.\]
\end{definition}

\begin{example}\label{ExGr24Hypersurface}
For $1\leq i_1<i_2\leq 4$, label the $\binom 42=6$ homogeneous coordinates of $\PP^5$ by $\Delta_{i_1,i_2}$. The corresponding Plücker coordinates on $\Gr(2,4)$ give a projective embedding of $\Gr(2,4)$ as a hypersurface in $\PP^5$, whose equation is
\[\Delta_{1,3}\Delta_{2,4} = \Delta_{1,2}\Delta_{3,4} + \Delta_{1,4}\Delta_{2,3}.\]
For example, the matrix
\[M=
\begin{bmatrix}
1 & 2 & 0 & 1\\
0 & 1 & -1 & 1
\end{bmatrix} 
\]
has $\Delta_{1,2}(M)=1$, $\Delta_{1,3}(M)=-1$, $\Delta_{1,4}(M)=1$, $\Delta_{2,3}(M)=-2$, $\Delta_{2,4}(M)=1$, and ${\Delta_{3,4}(M)=1}$. Note that row operations on $M$ change all Plücker coordinates by a common factor, which is immaterial once one thinks of them as homogeneous coordinates on $\PP^5$.

\end{example}

\subsection{Positroid strata}

The complex Grassmannian $\Gr(k,n)$ decomposes into disjoint subsets $\Pos_f$ labeled by bounded affine permutations $f\in\Bound(k,n)$, see \cite{KLS}. Any $M\in\Mat(k,n)$ with columns $M_1, \ldots, M_n$ extends periodically to a matrix with infinitely many columns, by setting $M_{i+n}=M_i$ for all $i\in\ZZ$. Define an associated function $f_M:\ZZ \to \ZZ$ by
\[
	f_M(i) = \min \{ j\geq i \suchthat M_i \in \Span(M_{i+1},\ldots ,M_j)\}.
\]

If $M\in\Mat(k,n)$ has rank $k$, then $f_M:\ZZ\to\ZZ$ is a $k$-bounded affine permutation of size $n$ that depends only on $[M]\in\Gr(k,n)$.

\begin{example}\label{ExGr24BAP}
The matrix $M\in\Mat(2,4)$ from \cref{ExGr24Hypersurface} extends periodically to a matrix with infinitely many columns
\[
	\begin{bmatrix}
	\cdots & 0 & 1 & \textbf{1} & \textbf{2} & \textbf{0} & \textbf{1} & 1 & 2 & \cdots \\
	\cdots & -1 & 1 & \textbf{0} & \textbf{1} & \textbf{-1} & \textbf{1} & 0 & 1 & \cdots
	\end{bmatrix}
	\quad ,
\]
and the corresponding bounded affine permutation $f_M:\ZZ\to\ZZ$ of type $(2,4)$ is $f_M=[3,4,5,6]$.
\end{example}

\begin{definition}\label{DefPositroidStratum}
The \emph{positroid stratum} associated to $f\in\Bound(k,n)$ is defined as
\[\Pos_f := \{ [M] \in \Gr(k,n) \suchthat f_M = f\}.\]
\end{definition}

The adjective positroid comes from the fact that the closure of a stratum is defined by the vanishing of Plücker coordinates $\Delta_{i_1,\ldots i_k}$ whose indexing sets ${\{i_1,\ldots ,i_k\}\subset \left\{1,\ldots,n\right\}}$ form a particular class of matroids \cite{postnikov2006total}. The term strata refers to the following property.

\begin{theorem}[Knutson--Lam--Speyer {\cite[Theorems 5.9, 5.10]{KLS}}]\label{ThmPositroidStratification}
Each positroid stratum is locally closed in the Zariski topology, and has closure
\[
	\cl(\Pos_f)= \bigcup_{f' \geq f} \Pos_{f'}.
\]
\end{theorem}

\begin{definition}[Partial order on positroid strata]
    Define $\Pos_1 \leq \Pos_2$ if and only if ${\Pos_1 \subset \cl(\Pos_2)}$.
\end{definition}
It follows immediately that $\leq$ defines a partial order on the set of positroid strata of $\Gr(k,n)$.
\begin{theorem}[{\cite[Theorem 5.9]{KLS}}]
    The codimension of $\Pos_f \subset \Gr(k,n)$ is equal to $\ell(f)$.
\end{theorem}

\begin{example}\label{ExGr24Strata}
There are 33 positroid strata $\Pos_f\subset\Gr(2,4)$: one of dimension 4, four of dimension 3, ten of dimension 2, twelve of dimension 1, and six of dimension 0.
Each dimension corresponds to a row in the Hasse diagram of \cref{fig:hasse_diagram_bound24}, with the bottom row containing the only top-dimensional stratum.
\end{example}

\section{Positroid links}\label{SecPositroidLinks}

	In this section we follow Casals--Gorsky--Gorsky--Simental \cite{CGGS} and associate a Legendrian link to a bounded affine permutation $f \in \Bound(k,n)$ (see \cref{SecBoundedAffinePermutations}).

  \begin{definition}
  Let $f \in \Bound(k,n)$ be a bounded affine permutation. For each $i\in \left\{1,\ldots,n\right\}$, let
		\[
			A_i(f) := \left\{(x,y) \in \RR^2 \suchthat (2x-f(i)-i)^2 + 4y^2 = (f(i)-i)^2\right\} \cap \left\{y \geq 0\right\} \subset \RR^2,
		\]
		be the upper semicircle of a circle intersecting the $x$-axis in the points $(i,0)$ and $(f(i),0)$. We define the \emph{juggling diagram} associated to $f$ to be the subset $\bigcup_{i=1}^n A_i(f) \subset \RR^2$.
  \end{definition}

		\begin{definition} \label{def:juggling_braid}
			Let $f\in \Bound(k,n)$ be a bounded affine permutation. After modifying the associated juggling diagram with the moves shown in \cref{fig:smoothing}, we obtain a tangle diagram. After enumerating the strands of the tangle diagram from top to bottom we can describe the tangle diagram with a braid word that we denote by $J_k(f)$ and call the \emph{juggling braid} of $f$.
			\begin{figure}[!htb]
				\centering
				\includegraphics{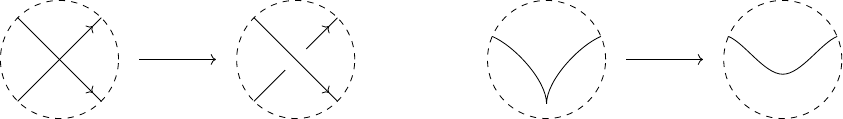}
				\caption{Converting from a juggling diagram to a braid via specified smoothings of cusps and crossings.}\label{fig:smoothing}
			\end{figure}
		\end{definition}

  		\begin{remark}
			By the definition of a bounded affine permutation, $J_k(f)$ is a positive braid on $k$ strands.
		\end{remark}

See \cref{FigJugglingBraids} for examples of juggling diagrams and their corresponding juggling braids.

We will now set some notation. Let $\sigma_1,\ldots,\sigma_{k-1}$ denote the Artin generators of the braid group and let $\Br_k^+$ be the submonoid of the braid group generated by non-negative powers of the Artin generators. We let $\Delta_k = (\sigma_1) (\sigma_2 \sigma_1) \cdots (\sigma_{k-1} \cdots \sigma_1)$ denote the positive half twist. Let $w_0$ denote the image of $\Delta_k$ in the projection from the braid group to the symmetric group.

    \begin{figure}[!htb]
			\centering
         \begin{subfigure}{.45\textwidth}
            \includegraphics[scale=0.9]{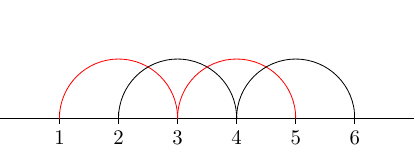}
            \caption{$J_2([3,4,5,6]) = \sigma_1^3$.}
        \end{subfigure}
        \begin{subfigure}{.45\textwidth}
            \includegraphics[scale=0.9]{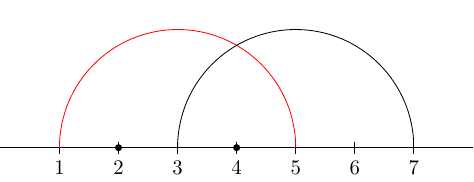}
            \caption{$J_2([5,2,7,4]) = \sigma_1$.}
            \label{fig:fixedPt}
          \end{subfigure}
  
		\begin{subfigure}[!htb]{\textwidth}
			\centering
			\includegraphics{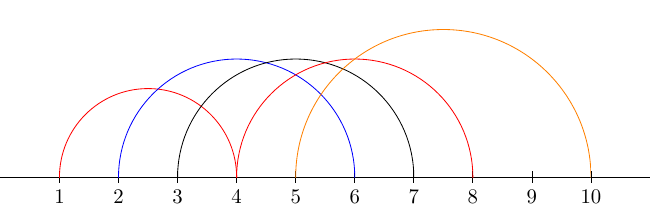}
			\caption{$J_4([4,6,7,8,10]) = \sigma_1\sigma_2\sigma_1 \sigma_3 (\sigma_2 \sigma_1)^2$. }
		\end{subfigure}
  
		\begin{subfigure}[!htb]{\textwidth}
			\centering
			\includegraphics{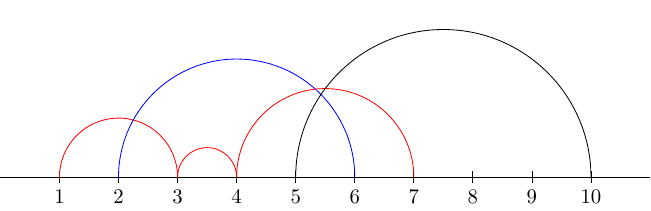}
			\caption{$J_3([3,6,4,7,10]) = \sigma_1 \sigma_2 \sigma_1^2$.}
   \end{subfigure}
   
   \caption{Examples of juggling diagrams and their corresponding juggling braid words. Note that the dots in \cref{fig:fixedPt} indicate fixed points of the bounded affine permutation. The strands are different colors to increase visual clarity.}
      \label{FigJugglingBraids}
   \end{figure}

		\begin{definition}[{\cite[Definition 3.3]{CGGS}}] \label{def:positroid_link}
			Let $f\in \Bound(k,n)$ be a bounded affine permutation, and let $J_k(f) \in \Br^+_k$ be its associated juggling braid. We define the \emph{positroid link} of $f$, denoted by $\Lambda_f$, to be the Legendrian $(-1)$-closure (see \cref{fig:closure}) of the positive braid $J_k(f)\Delta_k \in \Br^+_k$ with the orientation induced by giving all strands of $J_k(f) \in \Br^+_k$ the same orientation.
		\end{definition}
  
		\begin{figure}[H]
			\centering
			\includegraphics{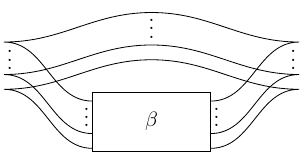}
			\caption{The front diagram of the Legendrian $(-1)$-closure of the positive braid word $\beta \in \Br^+_k$.}\label{fig:closure}
		\end{figure}

  \begin{remark}
			In \cite{CGGS}, there are other (
   Legendrian isotopic) descriptions of $\Lambda_f$, using other enumerations of positroid strata of the complex Grassmannian such as pairs of permutations (satisfying some properties), Le diagrams, and cyclic rank matrices. For the scope of this article, it suffices to consider juggling braids.
		\end{remark}

        \begin{lemma}\label{lma:tb_juggling_braid}
			Let $f\in \Bound(k,n)$. The Thurston--Bennequin number of $\Lambda_f$ is given by
			\[
				\tb(\Lambda_f) = |J_k(f)| - \frac{k(k+1)}{2},
			\]
			where $|J_k(f)|$ denotes the length of the braid word $J_k(f) \in \Br^+_k$.
		\end{lemma}
  
		\begin{proof}
			Recall that the Thurston--Bennequin number of a Legendrian link is given by the writhe minus the number of right cusps of a front diagram, see \cref{eq:thurston-bennequin}. There are $|J_k(f)| + |\Delta_k|$ positive crossings coming from the crossings in $\beta=J_k(f)\Delta_k$ and $k$ right cusps in the positroid link of $f$. Note that $|\Delta_k| = \frac{k(k-1)}{2}$. There are $2|\Delta_k| = k(k-1)$ negative crossings coming from the portion of the positroid link of $f$
   outside of $\beta=J_k(f)\Delta_k$ in the Legendrian $(-1)$-closure diagram. The sum of these contributions is
			\[
				\tb(\Lambda_f) = |J_k(f)| + \frac{k(k-1)}{2} - k(k-1)-k = |J_k(f)| - \frac{k(k+1)}{2}.
			\]
		\end{proof}
        \begin{corollary}\label{cor:tb}
			Let $f \in \Bound(k,n)$. The Thurston--Bennequin number of $\Lambda_f$ is given by
			\[
				\tb(\Lambda_f) = \dim \Pos_f + \# \Fix(f) - n.
			\]
            where $\Fix(f)=\{i\in\{1,\ldots, n\}~|~f(i)=i\}.$
		\end{corollary}
		\begin{proof}
            The statement of Lemma 3.10 in the first arXiv version of \cite{CGGS} states that 
            \[
                |J_k(f)| = |w| - |u| + \binom k2 - (n-k) + \#\Fix(f),
            \]
            where $(u,w)$ is a pair of permutations called the positroid pair corresponding to the bounded affine permutation $f$ (see \cite[Definition 2.2]{CGGS} and \cite[Proposition 3.15]{KLS}). It is well-known that $\dim \Pos_f = |w|-|u|$ (see e.g.\@ \cite[Corollary 3.2]{KLS2}), hence
            \begin{equation}\label{eq:cggs_old_lma}
                |J_k(f)| = \dim \Pos_f + \binom k2 -(n-k) + \#\Fix(f).
            \end{equation}
            Therefore we get
			\begin{align*}
				\tb(\Lambda_f) &\overset{\text{\cref{lma:tb_juggling_braid}}}{=} |J_k(f)| - \frac{k(k+1)}{2}\\
                & \overset{\text{\eqref{eq:cggs_old_lma}}}{=} \dim \Pos_f + \binom k2 - (n-k) + \# \Fix(f) - \frac{k(k+1)}{2}\\
                &= \dim \Pos_f + \# \Fix(f) - n.
			\end{align*}
		\end{proof}

  The main motivation for calling $\Lambda_f$ a positroid link is the following connection with positroid strata of the complex Grassmannian. To call upon this result, we follow the convention of placing a marked point on each strand in the braid $J_k(f) \Delta_k$, and place each to the right of all crossings in $J_k(f) \Delta_k$ on the respective strand when defining the augmentation variety associated to $\Lambda_f$, as in \cite[Section 2.6]{CGGS2}.
		\begin{theorem}[Casals--Gorsky--Gorsky--Simental \cite{CGGS2,CGGS}]\label{thm:aug_braid_var_iso}
			Let $f \in \Bound(k,n)$ be a bounded affine permutation, and consider its positroid link $\Lambda_f$. Then, there is an algebraic isomorphism
			\[
				\Pos_f \cong \Aug(\Lambda_f) \times (\CC^\ast)^{n-\# \Fix f-k}.
			\]
		\end{theorem}
  
		\begin{proof}
			Recall that we have one marked point in $\Lambda_f$ for each strand in the braid $J_k(f)\Delta_k$. Then by \cite[Theorem 2.30]{CGGS2} we have $\Aug(\Lambda_f) \cong X_0(J_k(f); w_0)$, where $X_0$ denotes the \emph{braid variety} as defined in \cite{CGGS2}. Then, by \cite[Theorem 1.3]{CGGS} we have \[\Pos_f \cong X_0(J_k(f);w_0) \times (\CC^\ast)^{n-\# \Fix f-k}\] which gives the result.
		\end{proof}

  \begin{proposition}
    For $f\in \Bound(k,n)$, the number of components of the link $\Lambda_f$ is given by the number of cycles of $f$ of length at least $2$ (see \cref{def:cycles}).
  \end{proposition}

\begin{proof}
	Consider a cyclic juggling diagram of $f$ which can be obtained from a juggling diagram by first restricting $\bigcup_{i=1}^n A_i(f) \subset \RR^2$ to $\left\{1 \leq x \leq n\right\}$ and then extending each arc cyclically. More precisely, we first arrange the juggling diagram of $f$
 so that no crossing of $\bigcup_{i=1}^n A_i(f)$ belongs to $\left\{x \geq n\right\} \subset \RR^2$ by a smooth isotopy of $\bigcup_{i=1}^n A_i(f)$ which leaves the braid word $J_k(f)$ unaffected (up to braid moves), see \cite[Lemma 2.19]{CGGS}. Then we define the cyclic juggling diagram of $f$ to be the subset
	\[
		\bar{A}(f) := \left(\bigcup_{i=1}^nA_i(f) \cap \left\{1 \leq x \leq n\right\}\right) \cup \left(\bigcup_{\{i \; \mid \; f(i) > n\}} A^{\text{shift}}_i(f)\cap \left\{1 \leq x \leq n\right\}\right) \subset \RR^2
	\]
	where
	\[
		A^{\text{shift}}_i(f) := \left\{(x,y) \in \RR^2 \suchthat (2(x+n-1)-f(i)-i)^2 + 4y^2 = (f(i)-i)^2\right\} \cap \left\{y \geq 0\right\} \subset \RR^2
	\]
	is the arc $A_i(f)$ shifted to the left by $n-1$. Then, each cycle of $f$ corresponds to a sequence of arcs that closes up onto itself in the cyclic juggling diagram. See \cref{fig:cyclicjugglingdiagram} for an example. We turn the cyclic juggling diagram $\bar{A}(f)$ into a braid by using the smoothing modifications of \cref{fig:smoothing} to obtain a cyclic juggling braid $\bar{J}_k(f)$. We take the $(-1)$-closure of the $\bar{J}_k(f)$. Then, the resulting link is smoothly isotopic to the $(-1)$-closure of the juggling braid $\Delta_k J_k(f)$, i.e.\@ the link $\Lambda_f$, see \cref{fig:halftwistisotopy}. 
	Thus the number of components of $\Lambda_f$ is exactly the number of cycles of $f$.

  \begin{figure}[!htb]
			\centering
			\includegraphics{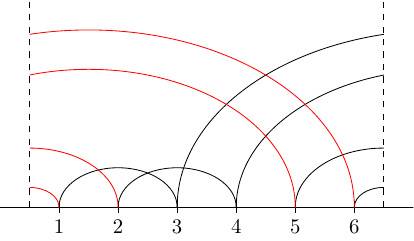}
			\caption{Extending the arcs of $[3, 4, 12, 11, 8, 7]$ cyclically to construct a cyclic juggling diagram}\label{fig:cyclicjugglingdiagram} 
   \end{figure}

   \begin{figure}[!htb]
			\centering
			\includegraphics{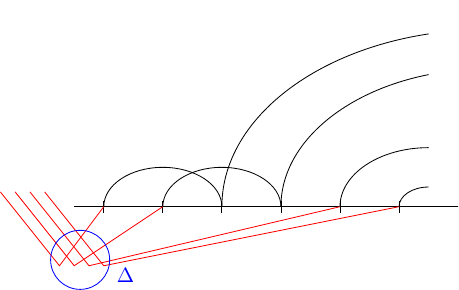}
			\caption{The result of a smooth isotopy from \cref{fig:cyclicjugglingdiagram}, where all the arcs are pulled downwards, demonstrating the half twist obtained from the added arcs of the cyclic juggling diagram.}\label{fig:halftwistisotopy}
   \end{figure}

\end{proof}

\section{Construction of the Lagrangian cobordisms}\label{SecCobordisms}

We say that there is a \emph{path} from $f$ to $g$ in $\Bound(k,n)$ if there is a sequence of affine bounded permutations $(h_1,\ldots,h_k)$ (this sequence might be empty) such that
\[
    f \lessdot h_1 \lessdot \cdots \lessdot h_k \lessdot g.
\]

\begin{theorem}\label{thm:main1}
	Given any path from $f$ to $g$ in $\Bound(k,n)$, there is an exact Lagrangian cobordism from $\Lambda_g$ to $\Lambda_f$.
\end{theorem}
 
\begin{proof}

Recall from \cref{def:juggling_braid} that any bounded affine permutation $f$ corresponds to a juggling braid $J_k(f)$ which corresponds by \cref{def:positroid_link} to a Legendrian $\Lambda_f$ given by the $(-1)$-closure of the positive braid $J_k(f)\Delta_k$. There is a convenient Lagrangian projection of $\Lambda_f$, see \cite[Figure 8]{casals2022braid}. Since the positive braid $J_k(f)\Delta_k$ contains a positive half twist $\Delta_k$, every crossing in the Lagrangian projection of $\Lambda_f$ corresponds to a contractible Reeb chord, see \cite[Proposition 2.8]{casals2022braid}. If two affine permutations $f$ and $g$ have the same juggling braids $J_k(f)=J_k(g)$, then $\Lambda_f = \Lambda_g$ by definition. Suppose now that two affine permutations $f$ and $g$ have juggling braids $J_k(f)$ and $J_k(g)$ such that $J_k(f)$ has one more positive crossing $x$ than $J_k(g)$. By Ng's resolution procedure, the crossing $x$ in the front projection corresponds to a contractible Reeb chord of $\Lambda_f$, we can perform a pinch move at $x$ as in \cref{fig:pinch} to obtain a Lagrangian saddle cobordism from $\Lambda_g$ to $\Lambda_f$.

Let $f,g\in\Bound(k,n)$ with $f<g$. It suffices to assume $f\lessdot g$, that is $g=\sigma_i \circ f$ or $g=f\circ \sigma_i$. Recall from \cref{lma:inverse_bap} that $h$ is a bounded affine permutation if and only if $-(-h)^{-1}$ is. Thus, because $g=f\circ \sigma_i$ is equivalent to $-(-f)^{-1} = \sigma_i \circ (-(-g)^{-1})$, it suffices to consider the case $g=\sigma_i \circ f$.

Namely, assume $g(a)=i+1$, $g(b)=i$, $f(a)=i$ and $f(b)=i+1$ for some $a,b,i\in\ZZ_{\geq1}$ such that $a<b$. We show that $J_k(f)$ has one more positive crossing than $J_k(g)$ does or $J_k(f) = J_k(g)$. Therefore, there is either an orientable exact Lagrangian saddle cobordism from $\Lambda_g$ to $\Lambda_f$, or the Legendrian links $\Lambda_g$ and $\Lambda_f$ are Legendrian isotopic and so are related by a trivial exact Lagrangian cobordism.

Since $g(b)=i$, we know $b\leq i$ so as $a<b$, we have $a<b\leq i$. If $b=i$, the respective juggling diagrams of $f$ and $g$ contain the arcs shown in \cref{fig:A2}. Thus we see that the juggling braids $J_k(f)$ and $J_k(g)$ are equal as braids. If $b<i$, the juggling diagrams of $f$ and $g$ contain the arcs shown in~\cref{fig:A1}, from which we can immediately conclude that the juggling braid $J_k(f)$ has one more crossing than the juggling braid $J_k(g)$. 
    \begin{figure}[!htb]
	   \centering
	   \includegraphics{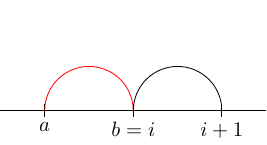}
        \hspace{1cm}
	   \includegraphics{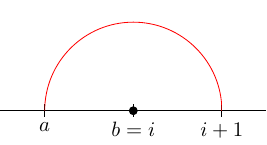}
	   \caption{Arcs in the juggling diagrams of $f$ (left) and $g$ (right) when $a<b=i$ where $f\lessdot g$.}\label{fig:A2}
    \end{figure}
    \begin{figure}[!htb]
	   \centering
	   \includegraphics{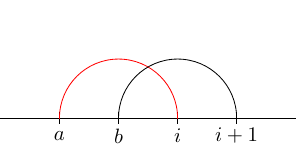}
        \hspace{1cm}
	   \includegraphics{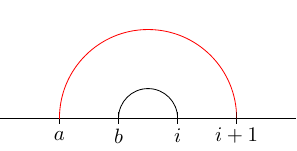}
	   \caption{Arcs in the juggling diagrams of $f$ (left) and $g$ (right) when $a<b<i$, where $f \lessdot g$.}\label{fig:A1}
    \end{figure}

\end{proof}
\begin{remark}\label{rmk:marked_pts}
In view of \cref{thm:aug_braid_var_iso}, a discussion on marked points in the construction of the exact Lagrangian cobordisms in the proof of \cref{thm:main1} is warranted. Since both $J_k(f)\Delta_k$ and $J_k(g)\Delta_k$ are $k$-stranded braids, their Legendrian $(-1)$-closures are decorated with one marked point per strand of the underlying braid. Any trivial exact Lagrangian cobordism remains trivial when taking marked points into account. Any saddle cobordism induced by a pinch move will involve newly created marked points in order to retain functoriality of the associated Chekanov--Eliashberg dg-algebras with coefficients in $\CC[t_1^{\pm 1},\ldots,t_k^{\pm 1}]$, see \cite[Section 3.5]{casals2022braid} and \cite[Section 2.4]{gao2020augmentations}. For our purpose, we will ignore the marked points created by pinch moves by evaluating them to $1$.
\end{remark}
\begin{theorem}\label{thm:main2}
		Given any path $\gamma$ from $f$ to $g$ in $\Bound(k,n)$, the corresponding exact Lagrangian cobordism $L_\gamma(g,f)$ from the proof of \cref{thm:main1} satisfies
		\[
			\chi(L_\gamma(g,f)) = \dim(\Pos_{g})-\dim(\Pos_{f})+\#\Fix(g)-\#\Fix(f),
		\]
  where $\Pos_{f}$ is the open positroid stratum associated to $f$.
\end{theorem}

\begin{proof}
Let $\gamma=(f_1,\ldots ,f_t)$ be a sequence of bounded affine permutations such that each pair of adjacent bounded affine permutations are related by an affine transposition, in other words, $f_i\lessdot f_{i-1}$ for all $1< i \leq t$. So $L_\gamma(f_t,f_1)$ is the corresponding Lagrangian cobordism. For each $i$, the dimensions of the respective positroid strata differ by 1. Then following the construction, the exact Lagrangian cobordism corresponding to $f_i\lessdot f_{i-1}$ is one of two things: 
\begin{enumerate}[1.]
\item A trivial cobordism, when the change in the juggling diagrams $J_{f_{i-1}}$ to $J_{f_{i}}$ is the creation of a fixed point. This contributes 1 to $\#\Fix(f_t)-\#\Fix(f_1)$.

\item A saddle cobordism corresponding to a single pinch move, when the change is the removal of a crossing. This contributes 1 to $\chi(L(\gamma))$.
\end{enumerate}

Thus, 
\[
\dim(\Pos_{f_1})-\dim(\Pos_{f_t}) = \chi(L(\gamma)) + \#\Fix(f_t) - \#\Fix(f_1).
\]
\end{proof}

\begin{remark}
\Cref{thm:main2} can also be proved using \cref{cor:tb} and the work of Chantraine \cite[Theorem 1.2]{chantraine2010lagrangian} which provides the change in the Thurston--Bennequin number for Legendrians related by an exact orientable Lagrangian cobordism. 
\end{remark}

\section{Examples}\label{SecExamples}

In \cref{ex:thmExample} we provide an example of \cref{ThmMain} and in \cref{ex:main_thm_not_equiv} a counterexample to the converse of \cref{ThmMain}.
	\begin{example}\label{ex:thmExample}
		We consider a path $f_1 \lessdot \cdots \lessdot f_7$ in the poset $\Bound(3,8)$ where the bounded affine permutations $f_1,\ldots, f_7$ are defined as follows
		\begin{align*}
			f_1 &= [5,4,7,6,8,9,10,11],& f_2 &= [5,4,8,6,7,9,10,11] \\
			f_3 &= [5,4,8,7,6,9,10,11],& f_4 &= [6,4,8,7,5,9,10,11] \\
			f_5 &= [6,3,8,7,5,9,10,12],& f_6 &= [6,2,8,7,5,9,11,12] \\
			f_7 &= [7,2,8,6,5,9,11,12].
		\end{align*}
		Each $f_i$ corresponds to a positroid stratum of $\Gr(3,8)$ of codimension $i+1$ and a Legendrian link $\Lambda_{f_i}$, each of which is the $(-1)$-closure of the corresponding juggling braid $J_3(f_i)$. We have listed the corresponding juggling braids below.
		\begin{align*}
			J_3(f_1) &= (\sigma_1\sigma_2)^4\sigma_2\sigma_1\sigma_2, & J_3(f_2) &= (\sigma_1\sigma_2)^3\sigma_2^2\sigma_1\sigma_2 \\
			J_3(f_3) &= (\sigma_1\sigma_2)^3\sigma_2\sigma_1\sigma_2, & J_3(f_4) &= (\sigma_1\sigma_2)^3\sigma_2\sigma_1\sigma_2\\
			J_3(f_5) &= (\sigma_1\sigma_2)^3\sigma_2\sigma_1, & J_3(f_6) &= (\sigma_1\sigma_2)^3\sigma_2\sigma_1 \\
			J_3(f_7) &= (\sigma_1\sigma_2)^3\sigma_1.
		\end{align*}
		We have depicted the corresponding composition of exact Lagrangian cobordisms in \cref{fig:lag_cob_ex}.
		\begin{figure}[!htb]
			\centering
			\includegraphics[scale=0.8]{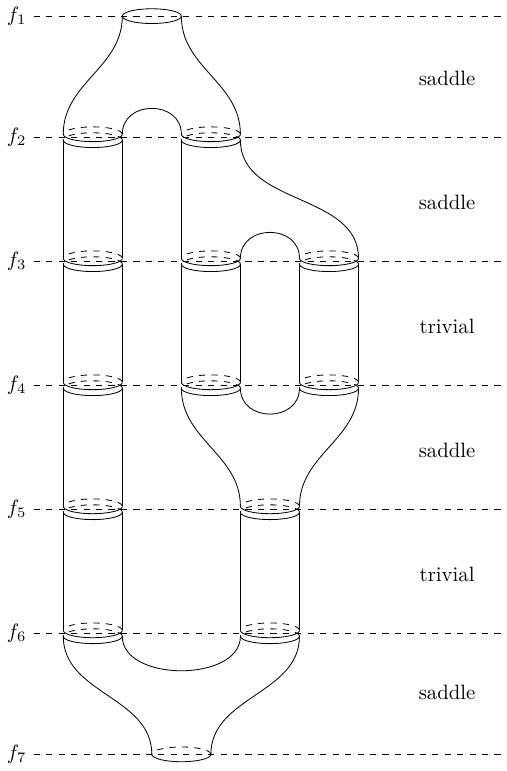}
			\caption{The exact Lagrangian cobordism from $\Lambda_{f_7}$ to $\Lambda_{f_1}$ corresponding to $f_1 \lessdot \cdots \lessdot f_7$.}\label{fig:lag_cob_ex}
		\end{figure}
		
  As noted above we have $\codim \Pos_{f_i} = i+1$. Because we see that $\#\Fix(f_1) = 0$ and ${\# \Fix(f_7) = 2}$, \cref{thm:main2} gives $\chi(L) = -4$, which correctly predicts that the exact Lagrangian cobordism depicted in \cref{fig:lag_cob_ex} has genus $2$.
	\end{example}
	\begin{example}\label{ex:main_thm_not_equiv}
        \begin{figure}[!htb]
			\centering
			\includegraphics{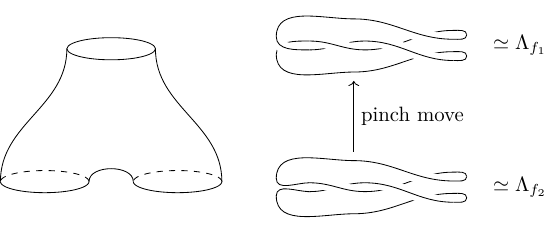}
			\caption{Pinch move giving a saddle cobordism from the Hopf link to the trefoil.}\label{fig:pinch_trefoil_hopf}
		\end{figure}
        We now show that the positroid links corresponding to two incomparable positroid strata can still be exact Lagrangian cobordant; this is the converse to \cref{ThmMain}.
  
		  Consider the two bounded affine permutations $f_1,f_2 \in \Bound(2,6)$ defined by
		\[
			f_1 := [3,4,5,7,8,6] \quad\text{and}\quad f_2 := [3,4,7,5,6,8].
		\]
		The corresponding juggling braids are $J_2(f_1) = \sigma_1^4$ and $J_2(f_2) = \sigma_1^3$. The two corresponding positroid links are the trefoil and the Hopf link, respectively. The two bounded affine permutations $f_1$ and $f_2$ correspond to two different positroid strata in $\Gr(2,6)$ of dimension $6$ and are therefore incomparable. However, there is an exact Lagrangian cobordism from $\Lambda_{f_2}$ to $\Lambda_{f_1}$ given by a saddle cobordism obtained by performing a pinch move at one of the crossings, see \cref{fig:pinch_trefoil_hopf}.
	\end{example}
    \begin{remark}\label{rmk:partial_converse}
        In \cref{ex:main_thm_not_equiv} we show that the braids $\sigma_1^4$ and $\sigma_1^3$ may appear as juggling braids of two incomparable bounded affine permutations. They also appear as the juggling braids $J_2(g_1)$ and $J_2(g_2)$ respectively for $g_1,g_2 \in \Bound(2,5)$ defined as
        \[
            g_1 := [3,4,5,6,7]\quad\text{and} \quad g_2 := [4,3,5,6,7],
        \]
        which \emph{are} comparable. Namely $g_1 \lessdot g_2$ because $g_2=g_1\circ\sigma_1$ and $\ell(g_1) = 0 < 1 = \ell(g_2)$.
    \end{remark}

\bibliographystyle{alpha}
\bibliography{LagCobPositroid}

\newcommand{\etalchar}[1]{$^{#1}$}
 \newcommand{\noop}[1]{}
\begin{thebibliography}{GLSBS22}

\bibitem[BGY06]{Brown_Goodearl_Yakimov_2006}
K.~A. Brown, K.~R. Goodearl, and M.~Yakimov.
\newblock Poisson structures on affine spaces and flag varieties. {I}. {M}atrix
  affine {P}oisson space.
\newblock {\em Adv. Math.}, 206(2):567--629, 2006.

\bibitem[CG22]{casals2022infinitely}
Roger Casals and Honghao Gao.
\newblock Infinitely many {L}agrangian fillings.
\newblock {\em Ann. of Math. (2)}, 195(1):207--249, 2022.

\bibitem[CGG{\etalchar{+}}22]{CGGLSS}
Roger Casals, Eugene Gorsky, Mikhail Gorsky, Ian Le, Linhui Shen, and Jos{\'e}
  Simental.
\newblock Cluster structures on braid varieties.
\newblock \url{https://arxiv.org/abs/2207.11607}, 2022.

\bibitem[CGGS20]{CGGS2}
Roger Casals, Eugene Gorsky, Mikhail Gorsky, and Jos\'{e} Simental.
\newblock Algebraic weaves and braid varieties.
\newblock \url{https://arxiv.org/abs/2012.06931}, 2020.

\bibitem[CGGS21]{CGGS}
Roger Casals, Eugene Gorsky, Mikhail Gorsky, and Jos\'{e} Simental.
\newblock Positroid links and braid varieties.
\newblock \url{https://arxiv.org/abs/2105.13948}, 2021.

\bibitem[Cha10]{chantraine2010lagrangian}
Baptiste Chantraine.
\newblock Lagrangian concordance of {L}egendrian knots.
\newblock {\em Algebr. Geom. Topol.}, 10(1):63--85, 2010.

\bibitem[Cha15]{chantraine2015lagrangian}
Baptiste Chantraine.
\newblock Lagrangian concordance is not a symmetric relation.
\newblock {\em Quantum Topol.}, 6(3):451--474, 2015.

\bibitem[Che02]{chekanov2002differential}
Yuri Chekanov.
\newblock Differential algebra of {L}egendrian links.
\newblock {\em Invent. Math.}, 150(3):441--483, 2002.

\bibitem[CN22]{casals2022braid}
Roger Casals and Lenhard Ng.
\newblock Braid loops with infinite monodromy on the {L}egendrian contact
  {DGA}.
\newblock {\em J. Topol.}, 15(4):1927--2016, 2022.

\bibitem[CW24]{Casals_Weng}
Roger Casals and Daping Weng.
\newblock Microlocal theory of {L}egendrian links and cluster algebras.
\newblock {\em Geom. Topol.}, 28(2):901--1000, 2024.

\bibitem[EGH00]{eliashberg2000introduction}
Y.~Eliashberg, A.~Givental, and H.~Hofer.
\newblock Introduction to symplectic field theory.
\newblock Number Special Volume, Part II, pages 560--673. 2000.
\newblock GAFA 2000 (Tel Aviv, 1999).

\bibitem[EHK16]{ekholm2016legendrian}
Tobias Ekholm, Ko~Honda, and Tam\'{a}s K\'{a}lm\'{a}n.
\newblock Legendrian knots and exact {L}agrangian cobordisms.
\newblock {\em J. Eur. Math. Soc. (JEMS)}, 18(11):2627--2689, 2016.

\bibitem[EN22]{etnyre2022legendrian}
John~B. Etnyre and Lenhard~L. Ng.
\newblock Legendrian contact homology in {$\mathbb{R}^3$}.
\newblock In {\em Surveys in differential geometry 2020. {S}urveys in
  3-manifold topology and geometry}, volume~25 of {\em Surv. Differ. Geom.},
  pages 103--161. Int. Press, Boston, MA, [2022] \copyright 2022.

\bibitem[FG09]{FG}
Vladimir~V. Fock and Alexander~B. Goncharov.
\newblock Cluster ensembles, quantization and the dilogarithm.
\newblock {\em Ann. Sci. \'{E}c. Norm. Sup\'{e}r. (4)}, 42(6):865--930, 2009.

\bibitem[FZ02]{FZ}
Sergey Fomin and Andrei Zelevinsky.
\newblock Cluster algebras. {I}. {F}oundations.
\newblock {\em J. Amer. Math. Soc.}, 15(2):497--529, 2002.

\bibitem[Gei08]{G}
Hansj\"{o}rg Geiges.
\newblock {\em An introduction to contact topology}, volume 109 of {\em
  Cambridge Studies in Advanced Mathematics}.
\newblock Cambridge University Press, Cambridge, 2008.

\bibitem[GHKK18]{GHKK}
Mark Gross, Paul Hacking, Sean Keel, and Maxim Kontsevich.
\newblock Canonical bases for cluster algebras.
\newblock {\em J. Amer. Math. Soc.}, 31(2):497--608, 2018.

\bibitem[GL21]{GL}
Pavel Galashin and Thomas Lam.
\newblock Positroids, knots, and {$q,t$}-{C}atalan numbers.
\newblock {\em S\'{e}m. Lothar. Combin.}, 85B:Art. 54, 12, 2021.

\bibitem[GL23]{GL2}
Pavel Galashin and Thomas Lam.
\newblock Positroid varieties and cluster algebras.
\newblock {\em Ann. Sci. \'{E}c. Norm. Sup\'{e}r. (4)}, 56(3):859--884, 2023.

\bibitem[GLSB23]{GLS}
Pavel Galashin, Thomas Lam, and Melissa Sherman-Bennett.
\newblock Braid variety cluster structures, {II}: general type.
\newblock \url{https://arxiv.org/abs/2301.07268}, 2023.

\bibitem[GLSBS22]{GLSS}
Pavel Galashin, Thomas Lam, Melissa Sherman-Bennett, and David Speyer.
\newblock Braid variety cluster structures, {I}: 3{D} plabic graphs.
\newblock \url{https://arxiv.org/abs/2210.04778}, 2022.

\bibitem[GSW24]{gao2020augmentations}
Honghao Gao, Linhui Shen, and Daping Weng.
\newblock Augmentations, {F}illings, and {C}lusters.
\newblock {\em Geom. Funct. Anal.}, 34(3):798--867, 2024.

\bibitem[Har92]{harris1992algebraic}
Joe Harris.
\newblock {\em Algebraic geometry}, volume 133 of {\em Graduate Texts in
  Mathematics}.
\newblock Springer-Verlag, New York, 1992.
\newblock A first course.

\bibitem[KLS13]{KLS}
Allen Knutson, Thomas Lam, and David~E. Speyer.
\newblock Positroid varieties: juggling and geometry.
\newblock {\em Compos. Math.}, 149(10):1710--1752, 2013.

\bibitem[KLS14]{KLS2}
Allen Knutson, Thomas Lam, and David~E. Speyer.
\newblock Projections of {R}ichardson varieties.
\newblock {\em J. Reine Angew. Math.}, 687:133--157, 2014.

\bibitem[Lus98]{Lusztig_1998}
G.~Lusztig.
\newblock Total positivity in partial flag manifolds.
\newblock {\em Represent. Theory}, 2:70--78, 1998.

\bibitem[NRS{\etalchar{+}}20]{NRSSZ}
Lenhard Ng, Dan Rutherford, Vivek Shende, Steven Sivek, and Eric Zaslow.
\newblock Augmentations are sheaves.
\newblock {\em Geom. Topol.}, 24(5):2149--2286, 2020.

\bibitem[Pan17]{Pan}
Yu~Pan.
\newblock Exact {L}agrangian fillings of {L}egendrian {$(2,n)$} torus links.
\newblock {\em Pacific J. Math.}, 289(2):417--441, 2017.

\bibitem[Pol91]{polterovich1991the}
L.~Polterovich.
\newblock The surgery of {L}agrange submanifolds.
\newblock {\em Geom. Funct. Anal.}, 1(2):198--210, 1991.

\bibitem[Pos06]{postnikov2006total}
Alexander Postnikov.
\newblock Total positivity, {G}rassmannians, and networks.
\newblock \url{https://arxiv.org/abs/math/0609764}, 2006.

\bibitem[Rie06]{Rietsch_2006}
K.~Rietsch.
\newblock Closure relations for totally nonnegative cells in {$G/P$}.
\newblock {\em Math. Res. Lett.}, 13(5-6):775--786, 2006.

\bibitem[Sco06]{S}
Joshua~S. Scott.
\newblock Grassmannians and cluster algebras.
\newblock {\em Proc. London Math. Soc. (3)}, 92(2):345--380, 2006.

\bibitem[STWZ19]{STWZ}
Vivek Shende, David Treumann, Harold Williams, and Eric Zaslow.
\newblock Cluster varieties from {L}egendrian knots.
\newblock {\em Duke Math. J.}, 168(15):2801--2871, 2019.

\bibitem[STZ17]{STZ}
Vivek Shende, David Treumann, and Eric Zaslow.
\newblock Legendrian knots and constructible sheaves.
\newblock {\em Invent. Math.}, 207(3):1031--1133, 2017.

\bibitem[SW21]{SW}
Linhui Shen and Daping Weng.
\newblock Cluster structures on double {B}ott-{S}amelson cells.
\newblock {\em Forum Math. Sigma}, 9:Paper No. e66, 89, 2021.

\end{thebibliography}

\end{document}